\documentclass{amsart}
 \usepackage{graphicx}
\usepackage{amssymb, color}\usepackage{amsmath}
\usepackage{amsfonts}
\usepackage{latexsym}
\makeatletter
 \def\LaTeX{\leavevmode L\raise.42ex
   \hbox{\kern-.3em\size{\sf@size}{0pt}\selectfont A}\kern-.15em\TeX}
\makeatother

\newcommand{\BibTeX}{{\rm B\kern-.05em{\sc
i\kern-.025emb}\kern-.08em\TeX}}
\newtheorem{col}{Corollary}[section]
\newtheorem{thm}{Theorem}[section]
\newtheorem{defn}{Definition}

\newtheorem{rem}[thm]{Remark}
\theoremstyle{defn}
\newtheorem{lem}[thm]{Lemma}

\newtheorem{remark}[thm]{Remark}
\newtheorem{lemma}[thm]{Lemma}

\numberwithin{equation}{section}

\def\EB{{\bf A}}
\def\FB{{\bf B}}
\def\AB{{\bf A}}
\def\HB{{\bf J}}

\begin{document}

\title[Function spaces,  Moduli of continuity,  Hardy-Steklov  operators ]{ Besov and Paley-Wiener spaces, Moduli of continuity  and Hardy-Steklov operators associated with the group $"ax+b"$}
\author{Isaac Z. Pesenson}
\address{Department of Mathematics, Temple University,
Philadelphia, PA 19122} \email{pesenson@temple.edu}


\begin{abstract}
We introduce and describe relations between  Sobolev, Besov and Paley-Wiener spaces associated with three representations of the Lie group $G$ of   affine transformations of the line, also known as the $ "ax + b” $ group. These representations are:  left and right regular representations and a representation in a space of functions defined on the half-line. The Besov spaces are described as  interpolation spaces between respective  Sobolev spaces in terms of the $K$-functional and in terms of a relevant moduli  of continuity.  By using a Laplace operators  associated with these representations  a scales of relevant Paley-Wiener spaces are developed and a corresponding  $L_{2}$-approximation theory is constructed in which our Besov spaces appear as approximation spaces.  Another description of our Besov spaces is given in terms of a frequency-localized Hilbert frames.  A Jackson-type inequalities are also proven.

\end{abstract}

\maketitle

\section{Introduction}

A substantial part of the classical harmonic analysis on Euclidean spaces  is dealing with such notions as  Sobolev and Besov spaces, Paley-Wiener (bandlimited) functions, $K$-functional, moduli of continuity, Hardy-Steklov smoothing operators.  
These topics and their numerous extensions and generalizations still attracting attention of many mathematicians:  \cite{BGS},  \cite{O},  \cite{FFP}, \cite{GOTT}, \cite{KT}, \cite{KR}, \cite{NRT1}, \cite{PSU}, \cite{Treb}-\cite{TW2}.  For the classical results see  \cite{A}, \cite{BL}, \cite{BB}, \cite{KPS}, \cite{N}, \cite{PS}, \cite{T}.

In \cite{KP},  \cite{Pes78}-\cite{Pes22} we proposed a development of  Sobolev, Besov, and Paley-Wiener  spaces along with a corresponding theories of interpolation and approximation in  Banach and Hilbert spaces in which a strongly continuous and bounded representation of a Lie group is given. 
The objective of the present paper is to apply our theory to three different representations of the Lie group $G$ of affine transformations of the line, also known as the $ "ax + b” $ group, which is of special interest in harmonic analysis. From one hand, we develop harmonic analysis in a new settings related to this group:     left and right regular representations and a representation in a space of functions defined on the half-line.   From other hand, by treating a concrete group we are able to deliver direct  and simple proofs for all our statements. These proofs are independent on our previous papers in which  notions of the general Lie theory were used.

The paper is organized as follows. Subsections 2.1 and 2.2 devoted to the group $ "ax+b" $ and its representations. 
Subsections \ref{L12}-\ref{X12} discuss respectively  left-regular, right-regular representations of $G$ and also its representation in certain spaces $X^{p}, \>1\leq p<\infty$ of functions on the half-line. They also contain definitions of the relevant Sobolev spaces and moduli of continuity. 
 Section \ref{G-fr-w} generalizes definitions and problems formulated in subsections \ref{L12}-\ref{X12} to the case of a strongly continuous bounded representation of $G$ in a Banach space ${\bf E}$. In particular, it contains the definition and properties of the mixed modulus of continuity  of order $r$:  $\>\Omega^{r}(s, f), \>f\in {\bf E},\>s>0, \>r\in \mathbb{N}$. In this section we also define Besov subspaces of ${\bf E}$ in terms of the $K$-functional and formulate  our main results Theorems \ref{Main-Ineq-000} and \ref{Main-000}. In section \ref{H-U} we treat the case of a unitary representation of $G$ in a Hilbert space ${\bf H}$. By using a corresponding self-adjoint Laplace operator a relevant analogs of  Paley-Wiener subspaces of ${\bf H}$ are defined. In subsection \ref{J} we formulate (Theorem \ref{main-Jack}) and discuss our version of a Jackson-type inequality which describes relations between  approximations by our Paley-Wiener functions and the moduli of  continuity $\>\Omega^{r}(s, f), \>f\in {\bf E},\>s>0, \>r\in \mathbb{N}$. After all, in section \ref{proofs} we prove our main Theorems  \ref{Main-Ineq-000} and \ref{Main-000}. The proof requires introduction of a generalization of the Hardy-Steklov smoothing operator. Appendix contains some basic information about Interpolation and Approximation spaces.

\section{The group $ax+b$ and its representations}\label{Group and Rep}

\subsection{The group $ "ax+b" $}\label{G-0}
The group $G$ of all linear transformations of $\mathbb{R}$ preserving orientation  (also known as the "$ax+b,\>a>0$, group") is isomorphic to the group of matrices of the form

$$
g(a,b)=\left(\begin{array}{|c|c|}\hline  a &  b \\\hline 0 & 1 \\\hline \end{array}\right),\>\>\>a>0, \>\>g(a,b)\in G.
$$
Every such matrix can be considered as  a linear transformation of $\mathbb{R}_{+}$ given by the formula $y=ax+b$ and if one has a transformation $y=a_{1}x+b_{1}$ followed by a transformation $z=a_{2}y+b_{2}$ then the resulting transformation is $x=a_{1}a_{2}x+a_{2}b_{1}+b_{2}$. In other words,
$$
g(a_{2}, b_{2})g(a_{1},b_{1})=g(a_{1}a_{2},\>a_{2}b_{1}+b_{2}).
$$
Consider  one-parameter subgroups 
\begin{equation}\label{subgroups}
\left\{\left(\begin{array}{|c|c|}\hline  e^{x} &  0 \\\hline 0 & 1 \\\hline \end{array}\right)\right\}_{x\in \mathbb{R}},\>\>\>\>\left\{\left(\begin{array}{|c|c|}\hline  1 &  x \\\hline 0 & 1 \\\hline \end{array}\right)\right\}_{x\in \mathbb{R}},
\end{equation}
 The matrices 
\begin{equation}\label{Algebra}
X_{1}=\left(\begin{array}{|c|c|}\hline  1 &  0 \\\hline 0 & 0 \\\hline \end{array}\right),\>\>\>X_{2}=\left(\begin{array}{|c|c|}\hline  0 &  1 \\\hline 0 & 0 \\\hline \end{array}\right),
\end{equation}
represent Lie algebra elements which are tangent to the above subgroups at the  identity $e\in G$. It means that for $x\in \mathbb{R}$ 
$$
\exp xX_{1}=\left(\begin{array}{|c|c|}\hline  e^{x} &  0 \\\hline 0 & 1 \\\hline \end{array}\right),\>\>\>\>\exp xX_{2}=\left(\begin{array}{|c|c|}\hline  1 &  x \\\hline 0 & 1 \\\hline \end{array}\right).
$$
In general, by using the following properties

$$
\left(\begin{array}{|c|c|}\hline  1 &  \alpha \\\hline 0 & 0 \\\hline \end{array}\right)^{n}=\left(\begin{array}{|c|c|}\hline  1 &  \alpha \\\hline 0 & 0 \\\hline \end{array}\right),\>\>\>\>n\in \mathbb{N},
$$
$$
\left(\begin{array}{|c|c|}\hline   0&  1 \\\hline 0 & 0 \\\hline \end{array}\right)^{n}=\left(\begin{array}{|c|c|}\hline  0 &  0\\\hline 0 & 0 \\\hline \end{array}\right),\>\>\>\>n>1,
$$
one obtains for $x_{1}, x_{2}\in \mathbb{R}$
$$
\exp (x_{1}X_{1}+x_{2}X_{2})=\left(\begin{array}{|c|c|}\hline  e^{x_{1}} & \frac{x_{2}}{x_{1}}(e^{x_{1}}-1) \\\hline 0 & 1 \\\hline \end{array}\right)
$$
In other words, the map
$$
Exp: (x_{1}, x_{2})\mapsto \exp (x_{1}X_{1}+x_{2}X_{2}),\>\>(x_{1}, x_{2})\in \mathbb{R}^{2},
$$ 
is a coordinate system in a neighborhood of $e\in G$.
One can also consider another coordinate system around identity $e\in G$
$$
\varphi: (x_{1}, x_{2})\mapsto \exp(x_{1}X_{1})\exp(x_{2}X_{2}),\>\>(x_{1}, x_{2})\in \mathbb{R}^{2}.
$$
Indeed, since
 \begin{equation}
\exp(x_{1}X_{1})\exp(x_{2}X_{2})=\left(\begin{array}{|c|c|}\hline  e^{x_{1}} &  0 \\\hline 0 & 1 \\\hline \end{array}\right)\left(\begin{array}{|c|c|}\hline  1 &  x_{2} \\\hline 0 & 1 \\\hline \end{array}\right)=\left(\begin{array}{|c|c|}\hline  e^{x_{1}} &  x_{2}e^{x_{1}} \\\hline 0 & 1 \\\hline \end{array}\right),  
\end{equation}
one can see that any given element of the group $G$
$$
\left(\begin{array}{|c|c|}\hline  a &  b \\\hline 0 & 1 \\\hline \end{array}\right)\in G,\>\>\>a>0, \>b\in \mathbb{R},
$$
 can be written as 
$$
\exp \left(\ln a \>X_{1}\right)\exp \left( \frac{b}{a}\> X_{2}\right).
$$
We also notice the following obvious relation 
\begin{equation}\label{lie}
\left[X_{1}, X_{2}\right]=X_{1}X_{2}-X_{2}X_{1}=X_{2}.
\end{equation}
The group $G$ can be identified with  the right half-plane 
\begin{equation}\label{half-pl}
\left\{(a,b) \>|\> a\in \mathbb{R}_{+},\>b\in \mathbb{R}\right\}=(0,\>\infty)\times \mathbb{R},
\end{equation}
equipped with the group operation 
\begin{equation}\label{multiplication}
(a, b)(c, d)=(ac,\>ad+b).
\end{equation}
In a such realization of $G$ every function $f$  on it can be treated as a function on the right half-plane in variables $(a,b)$  where $(a, b)\in (0,\>\infty)\times \mathbb{R}$. In this case  the left-invariant measure on $G$   is given by the formula 
\begin{equation}\label{left_measure}
d\mu_{l}=\frac{1}{a^{2}}da\>db.
\end{equation}
The group $G$ is not unimodular and the right-invariant measure on it is given by the formula
$$
d\mu_{r}=\frac{1}{a}da\>db.
$$
\subsection{Representations}\label{Rep}

Let us remind that a strongly continuous representation of a Lie group $G$ in a Banach space $\mathbf{E}$ is  a homomorphism $g\mapsto T(g),\>\>\>g\in \Gamma,\>\>\>T(g)\in GL(\mathbf{E})$, of $G$ into the group $GL(\mathbf{E}) $ of linear bounded invertible operators in $\mathbf{E}$ such that trajectory $T(g)f, \>\>g\in G,\>\>f\in \mathbf{E}, $ is continuous with respect to $g$ for every $f\in \mathbf{E}$. We will consider only uniformly bounded representations. In this case one can introduce a new norm $\|f\|^{'}_{\mathbf{E}}=\sup_{g\in G}\|T(g)f\|_{\mathbf{E}},\>\>f\in \mathbf{E},$ in which $\|T(g)f\|^{'}_{\mathbf{E}}\leq \|f\|^{'}_{\mathbf{E}}$.  Thus, without any restriction  we will assume that the last inequality  is satisfied in the original norm $\|\cdot\|_{\mathbf{E}}$.

Let  $X_{1}, X_{2}$ be the matrices in (\ref{Algebra}) which form a basis of the Lie algebra $\mathbf{g}$ of the group $G$.    
With each $X_{j}, \>j=1,2,$ one associates a  strongly continuous one-parameter group of isometries $t\mapsto T(\exp t X_{j}),\>\>t\in \mathbb{R},$ which will be also denoted as $T_{j}(t)=T(\exp tX_{j}),\>j=1,2.$ The generator of the group $T_{j}(t)=T(\exp tX_{j}),\>j=1,2,$ will be denoted as  $A_{j}, j=1,2,$ and the \textit{span} of operators $A_{j}, j=1,2,$ is known as the differential of the representation $T$. It is known \cite{N} that the common domain of  these operators contains the so-called Garding space $\mathcal{G}$. The Garding space  is a linear and dense in $\mathbf{E}$ manifold which comprised of all vectors  in ${\bf E}$, which have the form 
$$
\int_{G}\varphi(g)T(g)f d\mu_{l},\>\>f\in {\bf L}^{p}(G, d\mu_{l}),
\>\>\>\varphi\in C_{0}^{\infty}(G).
$$
The set $\mathcal{G}$ is  invariant with respect to operators $T(g)$ for all $g\in G$, and with respect to all polynomials in $A_{1}$ and $A_{2}$.

It will be convenient  identify the $ "ax+b" $ group with the group $G=\mathbb{R}_{+}\times \mathbb{R}$  equipped  with the 
 multiplication rule defined  by the formula
\begin{equation}\label{multiplication}
(a_{1}, b_{2})(a_{2}, b_{2})=(a_{1}a_{2}, a_{1}b_{2}+b_{1}),\>\>\>(a_{1}, b_{1}), (a_{2}, b_{2})\in G.
\end{equation}
 In these notations the subgroups in (\ref{subgroups}) are $\{(e^{x}, 0)\}, \> x\in \mathbb{R},$ and $\{(1,x)\}, \>x\in \mathbb{R},$ and 
$$
T_{1}(x)=T(e^{x}, 0)\in  GL({\bf E}),\>\>\>\>
T_{2}(x)=T(1, {x})\in  GL({\bf E}).
$$

\subsection{Left-regular representation of the group $ "ax+b "$}\label{L12} We consider the pair $(G, d\mu_{l})$ where $G$ is identified with the right-half plane (\ref{half-pl}) equipped with the multiplication (\ref{multiplication}), and $d\mu_{l}$ is the left-invariant measure (\ref{left_measure}). Let ${\bf L}^{p}(G, d\mu_{l}),\>1\leq p<\infty,$ be 
the corresponding Lebesgue space of functions $f: G\mapsto \mathbb{C}$.

The left-regular representation of $G$ in a space ${\bf L}^{p}(G, d\mu_{l}),\>1\leq p<\infty$  is defined by the formula 
\begin{equation}\label{L-rep}
U^{L}(a,b)f(x,y)=f\left((a,b)(x,y)\right)=f(ax, ay+b),
\end{equation}
where $(a,b), (x,y)\in  (0,\>\infty)\times \mathbb{R}.$
Since the measure $d\mu_{l}$ is left-invariant 
which means that for every integrable function $f$and every $h\in G$
\begin{equation}
\int_{G}f(hg)d\mu_{l}=\int_{G}f(g)d\mu_{l}
\end{equation}
every operator $U^{L}(a,b)$ is an isometry of the spaces ${\bf L}^{p}(G, d\mu_{l}),\>1\leq p<\infty$
\begin{equation}
\|U^{L}(a,b)f\|_{{\bf L}^{p}(G, d\mu_{l})}=\|f\|_{{\bf L}^{p}(G, d\mu_{l})},\>\>\>1\leq p<\infty.
\end{equation}
The corresponding 
Garding space $\mathcal{G}^{L}$ comprized of all functions in ${\bf L}^{p}(G, d\mu_{l}),\>1\leq p<\infty$, which have the form 
$$
\int_{G}\varphi(g)U^{L}(g)f d\mu_{l},\>\>f\in {\bf L}^{p}(G, d\mu_{l}),
\>\>\>\varphi\in C_{0}^{\infty}(G).
$$
It is linear, dense in ${\bf L}^{p}(G, d\mu_{l}),$ and invariant with respect to $U^{L}(g)$ for all $g\in G$ and with respect to all products $\mathbb{D}_{j_{1}}^{L}... \mathbb{D}_{j_{k}}^{L},\>k\in \mathbb{N}$ . 
The elements $g(e^{t},0),\>t\in \mathbb{R},$ form a one-parameter subgroup in $G$. This subgroup has a representation  as one-parameter $C_{0}$-group of operators acting in a corresponding ${\bf L}^{p}(G, d\mu_{l}),\>1\leq p<\infty,$ by the formula
\begin{equation}
U^{L}(e^{t}, 0)f(x,y)=U^{L}_{1}(t)f(x,y)=f((e^{t},0)(x,y))=f(e^{t}x,e^{t}y),
\end{equation}
and whose generator is defined as
\begin{equation}\label{DL}
\mathbb{D}^{L}_{1}f=\frac{d f(e^{t}x,e^{t}y)}{dt}|_{t=0}=\left(x\partial_{x}+y\partial_{y}\right)f,
\end{equation}
for all $f\in\mathcal{G}^{L}$. 
Another one-parameter subgroup is  formed by elements $g(1,t),\>t\in \mathbb{R}$, and the corresponding one-parameter $C_{0}$-group of operators in  ${\bf L}^{p}(G, d\mu_{l}),\>1\leq p<\infty$ is
\begin{equation}
U^{L}(1,t)f(x,y)=U_{2}^{L}(t)f(x,y)=f((1,t)(x,y))=f(x, y+t),
\end{equation}
whose generator is 
\begin{equation}
\mathbb{ D}^{L}_{2}f=\frac{d f(x, y+t)}{dt}|_{t=0}=\partial_{y}f,\>\>\>f\in \mathcal{G}^{L}.
\end{equation}
Both $\mathbb{D}_{1}^{L}, \mathbb{D}_{2}^{L}$, which originally defined on $\mathcal{G}^{L}$ admit closures from $\mathcal{G}^{L}$,  for which we will keep the same notations. 
Clearly,
\begin{equation}
[\mathbb{D}_{1}^{L}, \mathbb{D}_{2}^{L}]f=\mathbb{D}_{1}^{L}\mathbb{D}_{2}^{L}f-\mathbb{D}_{2}^{L}\mathbb{D}_{1}^{L}f=\mathbb{D}_{1}^{L}f,\>\>\>\>f\in\mathcal{G}^{L}.
\end{equation}
The next definition introduces an analog of the Sobolev spaces.
\begin{defn}
The Banach space $\mathbf{W}^{m}_{p}(\mathbb{D}_{1}^{L}, \mathbb{D}_{2}^{L}),\>m\in \mathbb{N},\> 1\leq p<\infty,$ is the set of functions $f$ in ${\bf L}^{p}(G, d\mu_{l})$ for which the following norm is finite

$$
|||f|||_{\mathbf{W}^{m}_{p}(\mathbb{D}_{1}^{L}, \mathbb{D}_{2}^{L}),}=\|f\|_{{\bf L}^{p}(G, d\mu_{l})}+\sum_{k=1}^{m}\sum_{(j_{1}, ... ,  j_{k})\in \{1,2\}^{k}}\|\mathbb{D}_{j_{1}}^{L}...\mathbb{D}_{j_{k}}^{L}f\|_{{\bf L}^{p}(G, d\mu_{l})}.
$$
\end{defn}
By using the closed graph theorem and the fact that each of $\mathbb{D}_{1}^{L}, \mathbb{D}_{2}^{L}$
is a closed operator in $L^{p}(G, d\mu_{l})$, one can show that this norm is equivalent to the norm
\begin{equation}\label{SobL}
\|f\|_{\mathbf{W}^{m}_{p}(\mathbb{D}_{1}^{L}, \mathbb{D}_{2}^{L}),}=\|f\|_{{\bf L}^{p}(G, d\mu_{l})}+\sum_{(j_{1}, ... ,  j_{m})\in \{1,2\}^{m}}\|\mathbb{D}_{j_{1}}^{L}...\mathbb{D}_{j_{m}}^{L}f\|_{{\bf L}^{p}(G, d\mu_{l})}.
\end{equation}
The mixed modulus of continuity is introduced as
\begin{equation}
\Omega_{m,p}^{L}( s, f)=
$$
$$
\sum_{1\leq j_{1},...,j_{m}\leq
2}\sup_{0\leq\tau_{j_{1}}\leq s}...\sup_{0\leq\tau_{j_{m}}\leq
s}\|
\left(U^{L}_{j_{1}}(\tau_{j_{1}})-I\right)...\left(U^{L}_{j_{m}}(\tau_{j_{m}})-I\right)f\|_{\bf L^{p}(G, d\mu_{l})},\label{M}
\end{equation}
where $f\in {\bf L}^{p}(G, d\mu_{l}),\ m\in \mathbb{N},  $ and $I$ is the
identity operator in ${{\bf L}^{p}(G, d\mu_{l})},\>\>1\leq p<\infty.$ 

When $p=2$ the representation $U^{L}$ is unitary. 
 In this  situation
 $\mathbb{D}_{1}, \mathbb{D}_{2}$ are skew-symmetric.  It is shown in \cite{Nelson} that in this case the non-negative Laplace operator
\begin{equation}\label{L}
\Delta_{L}=-\left(\mathbb{D}_{1}^{L}\right)^{2}-\left(\mathbb{D}_{2}^{L}\right)^{2}
\end{equation}
has a self-adjoint closure from $\mathcal{G}^{L}$. I The next statement foolows from Theorem \ref{H-equivalence}.

\begin{thm}\label{Sob-Mix}
The space $\mathbf{W}^{m}_{2}(\mathbb{D}_{1}^{L}, \mathbb{D}_{2}^{L})$ coincides with the domain $\mathcal{D}(\Delta_{L}^{m/2})$ and the norm (\ref{SobL}) is equivalent to the graph norm $\|f\|_{{\bf L}^{2}(G, d\mu_{l})}+\|\Delta_{L}^{m/2}f\|_{{\bf L}^{2}(G, d\mu_{l})}$. 
\end{thm}

\subsection{Right-regular representation of the group $ "ax+b "$}\label{R12}

The right-regular representation of $G$ in a space ${\bf L}^{p}(G, d\mu_{r}),\>1\leq p<\infty$  is defined by the formula 
\begin{equation}\label{L-rep}
U^{R}(a,b)f(x,y)=f\left((x, y)(a, b)\right)=f(xa, xb+y),
\end{equation}
where $(a,b), (x,y)\in  (0,\>\infty)\times \mathbb{R}.$
Since the measure $d\mu_{r}$ is right-invariant 
every operator $U^{R}(a,b)$ is an isometry of the spaces ${\bf L}^{p}(G, d\mu_{r}),\>1\leq p<\infty$
\begin{equation}
\|U^{R}(a,b)f\|_{{\bf L}^{p}(G, d\mu_{r})}=\|f\|_{{\bf L}^{p}(G, d\mu_{r})},\>\>\>1\leq p<\infty.
\end{equation}
Similarly to the situation with the left-regular representation we introduce the following one-parameter $C_{0}$-groups of operators acting in ${\bf L}^{p}(G, d\mu_{r}),\>1\leq p<\infty$

\begin{equation}
U^{R}(e^{t}, 0)f(x,y)=U^{R}_{1}(t)f(x,y)=f((x,y) (e^{t},0))=f(e^{t}x, y),
\end{equation}

\begin{equation}
U^{R}(1,t)f(x,y)=U_{2}^{R}(t)f(x,y)=f((x,y)(1,t))=f(x, xt+y),
\end{equation}
and their generators
\begin{equation}
\mathbb{D}^{R}_{1}f=x\partial_{x}f,\>\>\>\>\>\mathbb{ D}^{R}_{2}f=x\partial_{y}f, \>\>\>\>\>[\mathbb{D}^{R}_{1}, \mathbb{D}^{R}_{2}]f=\mathbb{D}^{R}_{1}f,\>\>\>\>f\in \mathcal{G}^{R}.
\end{equation}
Now one can introduce Sobolev spsces
\begin{equation}\label{SobR}
\|f\|_{\mathbf{W}^{m}_{p}(\mathbb{D}_{1}^{R}, \mathbb{D}_{2}^{R}),}=\|f\|_{{\bf L}^{p}(G, d\mu_{r})}+\sum_{(j_{1}, ... ,  j_{m})\in \{1,2\}^{m}}\|\mathbb{D}_{j_{1}}^{R}...\mathbb{D}_{j_{m}}^{R}f\|_{{\bf L}^{p}(G, d\mu_{r})},
\end{equation}
corresponding Garding space $\mathcal{G}^{R}$, and the self-adjoint non-negative Laplacian $\Delta_{R}$
in the space ${\bf L}^{2}(G, d\mu_{r})$
\begin{equation}\label{L}
\Delta_{R}=-\left(\mathbb{D}_{1}^{R}\right)^{2}-\left(\mathbb{D}_{2}^{R}\right)^{2}.
\end{equation}
An analog of Theorem \ref{Sob-Mix} also holds for this operator. 
The mixed modulus of continuity is defined as
\begin{equation}
\Omega_{m,p}^{R}( s, f)=
$$
$$
\sum_{1\leq j_{1},...,j_{m}\leq
2}\sup_{0\leq\tau_{j_{1}}\leq s}...\sup_{0\leq\tau_{j_{m}}\leq
s}\|
\left(U^{R}_{j_{1}}(\tau_{j_{1}})-I\right)...\left(U^{R}_{j_{m}}(\tau_{j_{m}})-I\right)f\|_{{\bf L}^{p}(G, d\mu_{r})},\label{M}
\end{equation}
where $f\in {\bf L}^{p}(G, d\mu_{r}),\ m\in \mathbb{N},  $ and $I$ is the
identity operator.

\subsection{Representation of the $ "ax+b" $ group in spaces ${\bf X}^{p}$}\label{X12}

For $p \in  [1, \infty)$, denote by $\| \cdot \|_{p} $ the norm of the Lebesgue space ${\bf L}^{p}(\mathbb{R}_{+}).$ 
The spaces ${\bf X}^{p}$ comprising all functions $f: \mathbb{R}_{+} \mapsto \mathbb{ C}$  such that $f(\cdot)(\cdot)^{-1/p}\in  {\bf L} ^{p}(\mathbb{R}_{+}) $ with the norm $\|f\|_{{\bf X}^{p} }:= \| f (\cdot)(\cdot)^{-1/p}\|_{p}$.

We  define a representation $U$ of $G$ on a space ${\bf X}^{p}, \>\>1\leq p< \infty$, by using the formula
$$
U(g)f(x)=e^{ib}f(ax),\>\>\>g=(a,b)\in G=\mathbb{R}_{+}\times\mathbb{R}.
$$
When $p=2$ the representation is unitary in the sense that every operator $U(g)$ is unitary with respect to the inner product 
$$
\langle f_{1}, f_{2}\rangle=\int_{0}^{\infty}f_{1}(x)\overline{f_{2}(x)}\frac{dx}{x}.
$$
The elements $g(e^{t},0),\>t\in \mathbb{R},$ form a one-parameter subgroup in $G$. This subgroup has a representation  as one-parameter $C_{0}$-group of operators acting in a corresponding ${\bf X}^{p}$ by the formula
$$
U_{1}(t)f(x)= f(e^{t}x), \>\>t\in \mathbb{R},\>\>1\leq p<\infty.
$$
The infinitesimal operator of this one-parameter group is defined on $C^{\infty}_{0}\cap {\bf X}^{p}$ by the formula 
$$
\mathbb{D}_{1}f(x)=\frac{d}{dt}f(e^{t}x)|_{t=0}=x\frac{d}{dx}f(x),\>\>\>f\in C^{\infty}_{0}\cap {\bf X}^{p},\>\>1\leq p<\infty.
$$
 The subset of elements of $G$ of the form  $g(1,t),\>t\in \mathbb{R},$  is another one-parameter subgroup which extends  to one-parameter $C_{0}$-group of operators acting in a corresponding ${\bf X}^{p}$ by the formula
$$
U_{2}(t)f(x)=e^{itx} f(x), \>\>t\in \mathbb{R},\>\>1\leq p<\infty.
$$
The corresponding infinitesimal operator is defined on $C_{0}^{\infty}\cap {\bf X}^p$ and is given by the formula
$$
\mathbb{D}_{2}f(x)=\frac{d}{dt}e^{itx}f(x)|_{t=0}=ixf(x),\>\>\>f\in  {\bf X}^{p},\>\>1\leq p<\infty.
$$ 
According to the general theory of $C_{0}$-groups in Banach spaces \cite{BB} both operators $\mathbb{D}_{1},  \mathbb{D}_{2}$  admit closures  in ${\bf X}^p$ for which we will use the same notations. 
We will use the notations  $\mathbb{D}_{1}$ for $x\frac{d}{dx}$ and $\mathbb{D}_{2}$ for multiplication by $ix$ in a space ${\bf X}^{p}(\mathbb{R}_{+},\> \frac{dx}{x}),\>1\leq p<\infty.$
\begin{defn}
The Banach space $\mathbf{W}^{m}_{p}(\mathbb{D}_{1}, \mathbb{D}_{2}),\>m\in \mathbb{N},\> 1\leq p<\infty,$ is the set of functions $f$ in ${\bf X}^{p}$ for which the following norm is finite
$$
\|f\|_{\mathbf{W}^{m}_{p}(\mathbb{D}_{1}, \mathbb{D}_{2})}=\|f\|_{{\bf X}^{p}}+\sum_{k=1}^{m}\sum_{(j_{1}, ... ,  j_{k})\in \{1,2\}^{k}}\|\mathbb{D}_{j_{1}}...\mathbb{D}_{j_{k}}f\|_{{\bf X}^{p}}.
$$
\end{defn}
By using closeness of the operators $\mathbb{D}_{j},\>j=1,2,$ one can check that the norm $\|f\|_{\mathbf{W}^{r}_{p}}$ is equivalent to the norm

$$
|||f|||_{\mathbf{W}^{r}_{p}(\mathbb{D}_{1}, \mathbb{D}_{2})}=\|f\|_{{\bf X}^{p}}+\sum_{(j_{1}, ... , j_{r})\in \{1,2\}^{r}} \|\mathbb{D}_{j_{1}}, ..., \mathbb{D}_{j_{r}} \|_{{\bf X}^{p}}.
$$
We note that the operators $\mathbb{D}_{1}, \mathbb{D}_{2}$ on $\mathbf{W}_{p}^{2}(\mathbb{D}_{1}, \mathbb{D}_{2})$ satisfy the relation
$$
\left[\mathbb{D}_{1}, \mathbb{D}_{2}\right]f=\mathbb{D}_{1}\mathbb{D}_{2}f-\mathbb{D}_{2}\mathbb{D}_{1}f=\mathbb{D}_{2}f,\>\>\>f\in \mathbf{W}_{p}^{2}(\mathbb{D}_{1}, \mathbb{D}_{2}),
$$
and they span a Lie algebra which is isomorphic to the Lie algebra of $G$. 
The mixed modulus of continuity is introduced as
\begin{equation}
\Omega_{m,p}( s, f)=
$$
$$
\sum_{1\leq j_{1},...,j_{m}\leq
2}\sup_{0\leq\tau_{j_{1}}\leq s}...\sup_{0\leq\tau_{j_{m}}\leq
s}\|
\left(U_{j_{1}}(\tau_{j_{1}})-I\right)...\left(U_{j_{m}}(\tau_{j_{m}})-I\right)f\|_{{\bf X}^{p}},\label{M}
\end{equation}
where $f\in {\bf X}^{p},\> m\in \mathbb{N},  $ and $I$ is the
identity operator in ${\bf X}^{p},\>\>1\leq p<\infty.$

In the space ${\bf X}^2$ we consider the corresponding Laplace operator $\Delta$ which is defined on the Garding space $\mathcal{G}$ by the formula
\begin{equation}\label{O}
\Delta_{U}=-\mathbb{D}_{1}^{2}-\mathbb{D}_{2}^{2}=-\left(x^{2}\frac{d^{2}}{dx^{2}}+x\frac{d}{dx}-x^{2}\right)
\end{equation}
which has self-adjoint closure from the corresponding Garding space.  Following Theorem is a consequence of Theorem \ref{H-equivalence}.

\begin{thm}\label{equivalence}
 The space $\mathbf{W}_{2}^{m}(\mathbb{D}_{1}, \mathbb{D}_{2}),\>m\in \mathbb{N},$ is isomorphic to the domain 
 $\mathcal{D}(\Delta_{U}^{m/2})$.
 \end{thm}

\section{General framework}\label{G-fr-w}

We considering a strongly continuous bounded representation of the group $ "ax+b" $ in a Banach space ${\bf E}$. 
and  using notations which were introduced in section \ref{Group and Rep}. 
\begin{defn}
The Banach space $\mathbf{E}^{r}=\mathbf{E}^{r}(\mathbb{A}_{1}, \mathbb{A}_{2}),\>r\in \mathbb{N},$ is the set of vectors $f$ in ${\bf E}$ for which the following norm is finite
\begin{equation}\label{Sob-1}
\|f\|_{\mathbf{E}^{r}}=\|f\|_{{\bf E}}+\sum_{k=1}^{r}\sum_{(j_{1}, ... ,  j_{k})\in \{1,2\}^{k}}\|\mathbb{A}_{j_{1}}...\mathbb{A}_{j_{k}}f\|_{{\bf E}}.
\end{equation}
\end{defn}
By using closeness of the operators $\mathbb{A}_{j}$ one can check that the norm $\|f\|_{{\bf E}^{r}}$ is equivalent to the norm
\begin{equation}\label{Sob-2}
|||f|||_{{\bf E}^{r}}=\|f\|_{{\bf E}}+\sum_{(j_{1}, ... , j_{r})\in \{1,2\}^{r}} \|\mathbb{A}_{j_{1}}, ..., \mathbb{A}_{j_{r}} \|_{{\bf E}}.
\end{equation}

 \subsection{$K$-functional and  modulus of continuity }

If $D$ generates in $\mathbf{E}$ a strongly continuous bounded semigroup $T_{D}(t),\> \|T(t)\|\leq 1, \>\>t\geq 0, $ then  the following functional is a natural generalization of the classical modulus of continuity (see \cite{BB})
$$
\omega^{r}_{D}( s, f)= 
\sup_{0\leq\tau\leq s}\|\left(T_{D}(\tau)-I\right)^{r}f\|_{\mathbf{E}},
$$
where $I$ is the
identity operator in ${\bf E}.$   By using the same reasoning as in the classical case (see \cite{Tim}, Ch. 3) one can establish  the following  inequalities 
\begin{equation}\label{ineq 1-1}
\omega_{D}^{r}(s, f)\leq s^{k}\omega_{D}^{r-k}(s, D^{k}f).
\end{equation}
\begin{equation}\label{ineq 2-1}
\omega_{D}^{r}\left(as, f\right)\leq \left(1+a\right)^{r}\omega_{D}^{r}(s, f),\>\> a>0,
\end{equation}
and
\begin{equation}\label{ineq 3-1}
s^{k}\omega_{D}^{r}\left(s, f\right)\leq c\left(s^{r+k}\|f\|_{\bf E}+\omega_{D}^{r+k}\left(s, f\right)\right),
\end{equation}
Now we are dealing with non-commuting one-parameter semigroups and our definition of the corresponding modulus of continuity is the following.
\begin{defn}
 The mixed modulus of continuity of a vector $f\in {\bf E}$ is introduced as
\begin{equation}\label{modulus}
\Omega^{r}( s, f)=
$$
$$
\sum_{1\leq j_{1},...,j_{r}\leq
2}\sup_{0\leq t_{j_{1}}\leq s}...\sup_{0\leq t_{j_{r}}\leq
s}\|
\left(T_{j_{1}}(t_{j_{1}})-I\right)...\left(T_{j_{r}}(t_{j_{r}})-I\right)f\|_{{\bf E}},\>\>r\in \mathbb{N}.
\end{equation}

\end{defn}
In what follows we will need the next two identities which can be easily veryfied.
\begin{lem}\label{Identities}
For any formal variables the following identities hold
\begin{equation}\label{first identity-1}
(a_{1}-1)a_{2}=(a_{1}-1)+(a_{1}-1)(a_{2}-1),
\end{equation}
 \begin{equation}\label{second identity-20}
 a_{1}a_{2}...a_{r}-1=(a_{1}-1)+a_{1}(a_{2}-1)+ ... +a_{1}a_{2} ... a_{r-1}(a_{r}-1).
  \end{equation}
\end{lem}
\begin{lem}\label{ineq}

The following inequalities hold  true 
\begin{equation}\label{ineq-1}
\Omega^{r}(s, f)\leq s^{k}C_{0}(r,k)\sum_{1\leq j_{1}, ...,   j_{k}\leq 2}\Omega^{r-k}(s, A_{j_{1}}... A_{j_{k}}f), 
\end{equation}
\begin{equation}\label{ineq-2}
\Omega^{r}\left(as, f\right)\leq C_{1}(a,r)\Omega^{r}(s, f),\>\> a>0,
\end{equation}
\begin{equation}\label{ineq-3}
s^{k}\Omega^{r}\left(s, f\right)\leq C_{2}(r,k)\left(s^{r+k}\|f\|_{\bf E}+\Omega^{r+k}\left(s, f\right)\right).
\end{equation}

\end{lem}
\begin{proof}
To prove the first item we observe that  for every one-parameter $C_{0}$-semigroup $T_{D}$ generated by $D$ the next following formula holds
$$
\left( T_{D}(s)-I\right)f=\int_{0}^{s}T_{D}(\tau)Dfd\tau,\>\>\>f\in \mathcal{D}(D).
$$
Then one has
$$
\Omega^{r}(s,f)\leq 
$$
$$
s\sum_{1\leq j_{1}, ..., j_{r}\leq 2}\sup_{0\leq t_{j_{1}}\leq s} ... \sup_{0\leq \tau \leq s}\|( T_{j_{1}}(t_{j_{1}})-I) ... ( T_{j_{r-1}}(t_{j_{r-1}})-I)T_{j_{r}}(\tau)A_{j_{r}}f\|_{\bf E}.
$$
Multiple applications of the identity
(\ref{first identity-1})
followed by applications of the triangle inequality, give the next estimate
$$
\Omega^{r}(s,f)\leq s\sum_{j=1, 2}\Omega^{r-1}(s, \mathbb{A}_{j}f)+s\sum_{j=1, 2}\Omega^{r}(s, \mathbb{A}_{j}f)\leq 3s\sum_{j=1, 2}\Omega^{r-1}(s, \mathbb{A}_{j}f).
$$
Continue this way we obtain
\begin{equation}
\Omega^{r}(s, f)\leq C_{0}(r,k)s^{k}\sum_{1\leq j_{1}, ...,   j_{k}\leq 2}\Omega^{r-k}(s, \mathbb{A}_{j_{1}}... \mathbb{A}_{j_{k}}f).
\end{equation}
To prove the second item we note first that if $0\leq a\leq b$ then
$$
\Omega^{r}\left(as, f\right)\leq \Omega^{r}\left(bs, f\right).
$$
In particular, it is true for $b=1$.  Next, if $a=n$ is a natural number then in the formula (\ref{modulus}) adapted to the case of $\Omega^{r} (ns, f)$, we replace every difference operator $(T_{j_{i}}(t_{j_{i}})-I)$ by
\begin{equation}\label{n-parent}
\left( \underbrace{T_{j_{i}}(\tau_{j_{i}}) ...T_{j_{i}}(\tau_{j_{i}})}_{n} -I\right)
\end{equation}
and accordingly, every $\sup_{0\leq t_{j_{i}}\leq ns}$ by $\sup_{0\leq \tau_{j_{i}}\leq s}$.
To every term (\ref{n-parent}) we apply the identity (\ref{second identity-20}).
  As a result, we will have
  $$
  \prod_{i=1}^{r}(T_{j_{i}}(t_{j_{i}})-I)f= \prod_{i=1}^{r} \left( \underbrace{T_{j_{i}}(\tau_{j_{i}}) ...T_{j_{i}}(\tau_{j_{i}})}_{n} -I\right)f=
   $$
   $$
    \prod_{i=1}^{r} \left[(T_{j_{i}}(\tau_{j_{i}})-I)+T_{j_{i}}(\tau_{j_{i}})(T_{j_{i}}(\tau_{j_{i}})-I)+ ...+ T_{j_{i}}((r-1)\tau_{j_{i}})  (T_{j_{i}}(\tau_{j_{i}})-I)\right]f.
   $$
  Multiplying out all the brackets (without opening parentheses) one will obtain a linear combinations of some products each of which will contain exactly $r$ differences of the form $(T_{j_{i}}(\tau_{j_{i}})-I)$. Every such a product will be handled  by  using (\ref{first identity-1}). Finally, it will get a representation of the  vector $  \prod_{i=1}^{r}(T_{j_{i}}(t_{j_{i}})-I)f$ as a linear combinations 
  of a terms of the form 
  $$
  T_{j_{i_{1}}}(\tau_{j_{i_{1}}}) ...     T_{j_{i_{l}}}(\tau_{j_{i_{l}}})   \left[      (T_{k_{\nu_{1}}}(\tau_{k_{\nu_{1}}})-I) ... (T_{k_{\nu_{r}}}(\tau_{k_{\nu_{r}}})-I)\right]f.
  $$
  Estimating each of such  terms by norm, and treating all the variables $\tau$  as independent variables varying between  $0$ and $s$, we obtain the inequality (\ref{ineq-2}).

  Let's prove (\ref{ineq-3}).  Introducing the notation 
  $$
  g_{  j_{2}, ..., j_{r}                         } =\left(T_{j_{2}}(t_{j_{2}})-I\right)...
\left(T_{j_{r}}(t_{j_{r}})-I\right)f,
$$ 
and then applying (\ref{ineq 3-1}) we are getting
\begin{equation}\label{modulus}
s^{k}\Omega^{r}( s, f)=
$$
$$
\sum_{1\leq j_{1},...,j_{r}\leq
2}\sup_{0\leq t_{j_{2}}\leq s}...\sup_{0\leq t_{j_{r}}\leq s}\left\{s^{k}\sup_{0\leq t_{j_{1}}\leq s}\|\left(T_{j_{1}}(t_{j_{1}})-I\right)
g_{  j_{2}, ..., j_{r}                         } \|\right\}\leq
$$
$$
\sum_{1\leq j_{1},...,j_{r}\leq
2}\sup_{0\leq t_{j_{2}}\leq s}...\sup_{0\leq t_{j_{r}}\leq s}\left\{s^{k+1}    \| g_{  j_{2}, ..., j_{r}                  }\|+     \sup_{0\leq t_{j_{1}}\leq s}\|\left(T_{j_{1}}(t_{j_{1}})-I\right)^{k+1}
g_{  j_{2}, ..., j_{r}                         } \|               \right\}\leq
$$
$$
\sum_{1\leq j_{1},...,j_{r}\leq
2}\sup_{0\leq t_{j_{2}}\leq s}...\sup_{0\leq t_{j_{r}}\leq s}s^{k+1}    \| g_{  j_{2}, ..., j_{r}                  }\|+ \Omega^{r+k}(s, f).
\end{equation}
 Continuing  this way we will obtain the inequality (\ref{ineq-3}). Lemma is proven.

\end{proof}

For the pair of Banach spaces $({\bf E}, \mathbf{E}^{r}),\>$  the $K$-functor is defined by the formula (see section \ref{Appendix}) 
$$
K(s^{r}, f,  {\bf E}, \mathbf{E}^{r})=\inf_{f=f_{0}+f_{1},f_{0}\in {\bf E},
f_{1}\in\mathbf{E}^{r}}\left(\|f_{0}\|_{{\bf E}}+s^{r}\|f_{1}\|_{\mathbf{E}^{r}}\right).\label{K}
$$
For the proof of the next Theorem see section \ref{proofs}.
\begin{thm}\label{Main-Ineq-000} There exist constants $c>0,\>C>0$, such that for all $f\in {\bf E},\>\>s\geq 0,$
\begin{equation}\label{main-ineq-000}
c \>\Omega^{r}(s, f)\leq K(s^{r}, f, {\bf E}, \>\mathbf{E}^{r})\leq C \left( \Omega^{r}(s, f)+\min(s^{r}, 1)\|f\|_{{\bf E}} \right)
\end{equation}
\end{thm}

\subsection{Besov spaces}

It is well known that most of remarkable properties of the so-called Besov functional spaces follow from the fact that they are interpolation spaces between two Sobolev spaces \cite{BB}, \cite{KPS}.  For this reason we define Besov spaces by the formula

\begin{defn}We introduce Besov spaces $\mathbf{B}^{\alpha}_{q}=\left(\mathbf{E},\>\mathbf{E}^{r}\right)^{K}_{\alpha/r, \>q},             \ 0<\alpha<r\in \mathbb{N},\ 1\leq\>q\leq \infty,$ as

\begin{equation}\label{interp}
 \mathbf{B}^{\alpha}_{q}=\left(\mathbf{E},\>\mathbf{E}^{r}\right)^{K}_{\alpha/r, \>q},             \ 0<\alpha<r\in \mathbb{N},\ 1\leq\>q\leq \infty,
\end{equation}
which means it is the space of all vectors in ${\bf E}$ with the norm
or to 
$$
\|f\|_{{\bf E}}+\left(\int_{0}^{\infty}\left(s^{-\alpha}K(s^{r}, f, {\bf E}, \>\mathbf{E}^{r})   \right)^{q}\frac{ds}{s}\right)^{1/q} , \>\>\>1\leq q<\infty,
$$
with the usual modifications for $q=\infty$.

\end{defn}

The following statement is an immediate consequence of Theorem \ref{Main-Ineq-000}.
\begin{thm}

 The Besov space $\mathbf{B}^{\alpha}_{q}$ coincides with the interpolation space $\left({\bf E}, \>\mathbf{E}^{r}\right)^{K}_{\alpha/r,\>q},             \ 0<\alpha<r\in \mathbb{N},\>1\leq  q\leq \infty,$ and its norm (\ref{Bnorm1}), (\ref{Bnorm2}) is equivalent to either of the following norms
\begin{equation}\label{Bnorm3}
\|f\|_{{\bf E}}+\left(\int_{0}^{\infty}(s^{-\alpha}\Omega^{r}(s,
f))^{q} \frac{ds}{s}\right)^{1/q} , \>\>\>1\leq q<\infty,
\end{equation}
with the usual modifications for $q=\infty$.

\end{thm}

The next result will be proven in section \ref{proofs}.

\begin{thm}\label{Main-000} The following holds true.

\begin{enumerate}

\item The following isomorphism holds true
\begin{equation}\label{reit}
\left({\bf E}, \>\mathbf{E}^{r}\right)^{K}_{\alpha/r,\>q}=\left(\mathbf{E}^{k_{1}},\>\mathbf{E}^{k_{2}}\right)^{K}_{(\alpha-k_{1})/(k_{2}-k_{1}),\>q},
\end{equation}
where $0\leq k_{1}<\alpha<k_{2}\leq r\in \mathbb{N},\>\>1\leq  q\leq \infty.$

\item
For $\alpha\in \mathbb{R}_{+} \>\>1\leq q\leq \infty,$ and $\alpha$ is not integer the Besov space $\mathbf{B}^{\alpha}_{q}$  coincides  with the subspace in ${\bf E}$ of all vectors for which the following norm is finite
\begin{equation}\label{Bnorm1}
\|f\|_{\mathbf{E}^{[\alpha]}}+\sum_{1\leq j_{1},...,j_{[\alpha] }\leq 2}
\left(\int_{0}^{\infty}\left(s^{[\alpha]-\alpha}\Omega^{1}
(s, \>\mathbb{A}_{j_{1}}...\mathbb{A}_{j_{[\alpha]}}f)\right)^{q}\frac{ds}{s}\right)^{1/q},
\end{equation}
where $[\alpha]$ is the integer part of $\alpha$.

\item In the case when 
$\alpha=k\in \mathbb{N}$ is an integer  the Besov space $\mathbf{B}^{\alpha}_{q}, \>\>1\leq q\leq \infty,$ coincides with the subspace in ${\bf E}$ of all the vectors for which the following norm is finite
 (Zygmund condition)
\begin{equation}\label{Bnorm2}
\|f\|_{\mathbf{E}^{k-1}}+ \sum_{1\leq j_{1}, ... ,j_{k-1}\leq 2}
\left(\int_{0}^{\infty}\left(s^{-1}\Omega^{2}(s,\>\mathbb{A}_{j_{1}}...\mathbb{A}_{j_{k-1}}f)\right)
 ^{q}\frac{ds}{s}\right)^{1/q}.
\end{equation}

\end{enumerate}

\end{thm}

 \section{The case of a unitary representation }\label{H-U}
 
 We keep the notations from the previous section, but assume that $T$ is a unitary representation of the $ "ax+b" $ group $G$ in a Hilbert space ${\bf H}$. As it was mentioned above, the operator 
 $$
 \Delta=-\mathbb{A}_{1}^{2}-\mathbb{A}_{2}^{2}
 $$
admits a self-adjoint non-negative closure from the corresponding Garding space.
The proof of the following theorem is given in \cite{Pes90a}, \cite{Pes22}.
\begin{thm}\label{H-equivalence}
The space $\mathbf{H}^{r}=\mathbf{H}^{r}(\mathbb{A}_{1}, \mathbb{A}_{2})$ with the norm (\ref{Sob-1}) is isomorphic to the domain $\mathcal{D}(\Delta^{r/2})$ equipped with the graph norm.

\end{thm}

\subsection{Paley-Wiener vectors in ${\bf H}$}

In the next definition we introduce the  Paley-Wiener vectors  associated with $\Delta$.

\begin{defn}\label{PWvector}
We say that a vector $f \in {\bf H}$ belongs to the  Paley-Wiener space ${\bf PW}_{\omega}\left(\Delta^{1/2}\right)$
 if and only if for every $g\in {\bf H}$ the scalar-valued function of the real variable  $ t \mapsto
\langle e^{it\Delta}f,g\rangle$ has an extension to the complex
plane as an entire function of the exponential type $\omega$.

\end{defn}
The next theorem contains generalizations of several results
from  classical harmonic analysis (in particular  the
Paley-Wiener theorem). It follows from our  results in
\cite{Pes00}.
\begin{thm}\label{PWproprties}
The following statements hold:
\begin{enumerate}

\item  the space ${\bf PW}_{\omega}\left(\Delta^{1/2}\right)$ is a linear closed subspace in
${\bf H}$,
\item the space 
 $\bigcup _{ \sigma >0}{\bf PW}_{\omega}\left(\Delta^{1/2}\right)$
 is dense in ${\bf H}$.

\item the space 
 ${\bf PW}_{\omega}\left(\Delta^{1/2}\right)$  is the image space of the projection operator ${\bf 1}_{[0,\>\omega]}(\Delta^{1/2})$ (to be understood in the sense of operational calculus). 

\item (Bernstein inequality)   $f \in {\bf PW}_{\omega}\left(\Delta^{1/2}\right)$ if and only if
$ f \in \mathcal{D}^{\infty}(\Delta)=\bigcap_{k=1}^{\infty}\mathcal{D}(\Delta^{k})$,
and the following Bernstein-type inequalities  holds true
\begin{equation}\label{Bern0}
\|\Delta^{s/2}f\|_{{\bf H}} \leq \omega^{s}\|f\|_{{\bf H}}  \quad \mbox{for all} \, \,  s\in \mathbb{R}_{+};
\end{equation}
\item (Riesz-Boas  interpolation formula) $f \in {\bf PW}_{\omega}\left(\Delta^{1/2}\right)$ if
and only if  $ f \in \mathcal{D}^{\infty}(\Delta^{\infty})$ and the
following Riesz-Boas interpolation formula holds for all $\omega> 0$:
\begin{equation}  \label{Rieszn}
i\sqrt{\Delta}f=\frac{\omega}{\pi^{2}}\sum_{k\in\mathbb{Z}}\frac{(-1)^{k-1}}{(k-1/2)^{2}}
e^{i\left(\frac{\pi}{\omega}(k-1/2)\right)\sqrt{\Delta}}f.
\end{equation}
\end{enumerate}
\end{thm}

\subsection{A Jackson-type inequality}\label{J}

We are using  the Schr\"{o}dinger  group $e^{it\Delta}$ to  introduce  the modulus of continuity
\begin{equation}\label{mod-L}
\omega_{\Delta}^{r}(t,f)=\sup_{0\leq \tau\leq t}\left\|\left(e^{it\Delta}-I\right)^{r}f\right\|_{{\bf H}}.
\end{equation}
The best approximation functional $\mathcal{E}_{\Delta}(s,f)$ is defined as
$$
\mathcal{E}_{\Delta}(\sigma, f)=\inf_{g\in {\bf PW}_{\sigma}(\Delta^{1/2})}\|f-g\|_{{\bf H}}.
$$
In \cite{Pes22} we proved  the following  Jackson-type estimate which holds for any self-adjoint operator $\mathcal{L}$
\begin{equation}\label{Schrod}
\mathcal{E}_{\mathcal{L}}(\sigma, f)\leq C(\mathcal{L})\omega_{\mathcal{L}}^{r}(\sigma^{-1}, f).
\end{equation}
Using the following well known inequalities which hold for every generator of a bounded $C_{0}$-semigroup (see \cite{BB}, Ch. 3), we can write
\begin{equation}\label{K-L-ineq}
\omega_{\Delta}^{r}(s, f)\leq c_{1}K\left(s^{r},f, {\bf H}, \mathcal{D}(\Delta^{r/2})\right)\leq C_{1}\left(\omega_{\Delta}^{r}(s, f)+\min(s^{r}, 1) \|f\|_{{\bf H}}\right),
\end{equation}
where $\mathcal{D}(\Delta^{r/2}) $ is  the domain  of the operator $\Delta^{r/2}$ with the graph norm $\|f||_{{\bf H}}+\|\Delta^{r/2}f\|_{{\bf H}}$. In addition, 
Theorem \ref{H-equivalence} implies  existence of a constant $C_{2}>0$ such that 
 \begin{equation}\label{C3}
 K\left(s^{r},f, {\bf H}, \mathcal{D}(\Delta^{r/2})\right)\leq C_{2}K\left(s^{r},f, {\bf H}, \mathbf{H}^{r}\right),\>\>\>f\in {\bf H}.
 \end{equation}
Thus we obtain 
$$
\mathcal{E}_{\Delta}(\sigma, f)\leq C(\Delta)\omega_{\Delta}^{r}(\sigma^{-1}, f)\leq C(\Delta)c_{1}K\left(\sigma^{-r},f, {\bf H}, \mathcal{D}(\Delta^{r/2})\right)\leq 
$$
$$
C(\Delta)c_{1}C_{2}K\left(\sigma^{-r},f, {\bf H}, \mathbf{H}^{r}\right)\leq 
$$
$$
C(\Delta)c_{1}C_{2}C_{1}\left(\Omega_{r}(\sigma^{-1},f)+ \min(\sigma^{-r}, 1)\|f\|_{{\bf H}}\right),
$$
where Theorem \ref{H-equivalence} was used. Now we can formulate our Jackson-type theorem.

\begin{thm}\label{main-Jack}There exists a constant $C>0$ which is independent on $f\in {\bf H}$ such that
\begin{equation}\label{main}
\mathcal{E}_{\Delta}(\sigma, f)\leq C\left(\Omega_{r}( \sigma^{-1}, f)+\min (\sigma^{-r}, 1)\|f\|_{\bf H}\right).
\end{equation}

\end{thm}

\begin{remark}
It is important to notice that since $\Omega_{r}( \tau, f)$ cannot be of order $o(\tau^{r})$ when $\tau\rightarrow 0$ (unless $f$ is invariant), the behavior of the right-hand side in (\ref{main}) is determined by the first term when $\sigma\rightarrow \infty$. In particular, if $f\in \mathbf{H}^{r}$, then due to the inequality
$$
\Omega^{r}(s, f)\leq C_{0}(r,k)s^{-k}\sum_{1\leq j_{1},...,j_{k}\leq 2}\Omega^{r-k}(s, \mathbb{A}_{j_{1}}...\mathbb{A}_{j_{k}}f),\>\>\>\>\>0\leq k\leq r,
$$
one has the best possible estimate
$$
\mathcal{E}_{\Delta}(\sigma, f)\leq C\Omega_{r}(\sigma, f)\leq C\sigma^{-r}\|f\|_{\mathbf{H}^{r}}.
$$

\end{remark}

\section{Comparison with the previous examples}

In subsection \ref{L12}  we have ${\bf E}= L^{p}(\mathbb{R}_{+}\times \mathbb{R}, \>\frac{1}{x^{2}}dx\>dy),\>1\leq p<\infty$. For $g\in G$ with $g=(a,b)\in \mathbb{R}_{+}\times\mathbb{R}$, 
$$
T(g)f(x,y)=U^{L}(a,b)f(x,y)=f(ax, ay+b),
$$
$$
T_{1}(t)f(x,y)=U_{1}^{L}(t)f(x,y)=f(e^{t}x, e^{t}y),\>\>\>
$$
$$
T_{2}(t)f(x,y)=U_{2}^{L}(t)f(x,y)=f(x, y+t),
$$
$$
\mathbb{A}_{1}=\mathbb{D}_{1}^{L}=x\partial_{x}+y\partial_{y},\>\mathbb{A}_{2}=\mathbb{D}_{2}^{L}=\partial_{y}. 
$$
The Hilbert case corresponds to $p=2$ and
$$
\Delta=\Delta^{L}=-(x\partial_{x}+y\partial_{y})^{2}-(\partial_{y})^{2}=
$$
$$
-(1+y^{2})\partial_{yy}-x^{2}\partial_{xx}-2xy\partial_{xy}-x\partial_{x}-y\partial_{y}.
$$
In subsection \ref{R12}  we have ${\bf E}= L^{p}(\mathbb{R}_{+}\times \mathbb{R}, \>\frac{1}{x}dx\>dy),\>1\leq p<\infty$.  For $g\in G$ with $g=(a,b)\in \mathbb{R}_{+}\times\mathbb{R}$, 
$$
T(g)f(x,y)=U^{R}(a,b)f(x,y)=f(ax, bx+y),
$$
$$
T_{1}(t)f(x,y)=U_{1}^{R}(t)f(x,y)=f(e^{t}x, y),
$$
$$
T_{2}(t)f(x,y)=U_{2}^{R}(t)f(x,y)=f(x, tx+t),
$$
$$
\mathbb{A}_{1}=\mathbb{D}_{1}^{R}=x\partial_{x},\>\mathbb{A}_{2}=\mathbb{D}_{2}^{R}=x\partial_{y}. 
$$
The Hilbert case corresponds to $p=2$ and
$$
\Delta=\Delta^{R}=-(x\partial_{x})^{2}-(x\partial_{y})^{2}=-x^{2}(\partial_{xx}+\partial_{yy})-x\partial_{x}.
$$
In subsection \ref{X12} we have $\>{\bf E}={\bf X}^{p}={\bf L}^{p}\left(\mathbb{R}_{+}, \frac{dx}{x}\right),\>1\leq p<\infty.$ For $g\in G$ with $g=(a,b)\in \mathbb{R}_{+}\times\mathbb{R}$, 
$$
T(g)f(x)=U(a,b)f(x)=e^{ib}f(ax),
$$
$$
T_{1}(t)f(x)=U_{1}(t)f(x)=f(e^{t}x),
$$
$$
T_{2}(t)f(x)=U_{2}(t)f(x)=e^{itx}f(x),
$$
$$
\mathbb{A}_{1}=\mathbb{D}_{1}=x\frac{d}{dx},\>\mathbb{A}_{2}=\mathbb{D}_{2}=ix. 
$$
The Hilbert case corresponds to $p=2$ and
\begin{equation}\label{oscillator}
\Delta=-\left(x\frac{d}{dx}\right)^{2}-(ix)^{2}=-\left(x\frac{d}{dx}\right)^{2}+x^{2}=-\left(x^{2}\frac{d^{2}}{dx^{2}}+x\frac{d}{dx}-x^{2}\right)
\end{equation}
\begin{rem}Note, that the classical harmonic oscillator has the form $-\left(\frac{d}{dx}\right)^{2}+x^{2}$
Since in the Mellin analysis the operator $x\frac{d}{dx}$ plays the same role as the operator $\frac{d}{dx}$ in the regular setting, our operator $\Delta$ in (\ref{oscillator}) can be treated as a Mellin harmonic oscillator.
\end{rem}

\section{Proofs of Theorems \ref{Main-Ineq-000} and \ref{Main-000} }\label{proofs}

We are using the same notations as above.

\begin{lem} The following formula holds for $\> m\in \mathbb{N}\cup \{0\}$:

\begin{equation}\label{A1A2}
\mathbb{A}_{2}^{m}T_{1}(t_{1})T_{2}(t_{2})f=e^{-mt_{1}}T_{1}(t_{1})T_{2}(t_{2})\mathbb{A}_{2}^{m}f,\>\>\>f\in \mathcal{G}.
\end{equation}

\end{lem}

\begin{proof}
Since 
$$
(e^{t_{1}}, 0)(1, t_{2})=(e^{t_{1}}, t_{2}e^{t_{1}}),
$$
and for any $(a,b)\in G=\mathbb{R}_{+}\times \mathbb{R}$
$$
(a, b)=(e^{\ln a}, 0)(1, b/a)
$$
one has
$$
\mathbb{A}_{2}T_{1}(t_{1})T_{2}(t_{2})f=\frac{d}{d\tau}T_{2}(\tau)T_{1}(t_{1})T_{2}(t_{2})f|_{\tau=0}=
$$
$$
\frac{d}{d\tau}T\left((1,\tau)      (e^{t_{1}}, 0)(1, t_{2})                  \right)f|_{\tau=0}= 
 \frac{d}{d\tau}T\left((e^{t_{1}}, \tau+t_{2}e^{t_{1}}) \right)f|_{\tau=0} =
 $$
 $$
  \frac{d}{d\tau}T\left(   (e^{t_{1}}, 0)(1, t_{2}+\tau e^{-t_{1}})       \right)f|_{\tau=0}=
\frac{d}{d\tau}T_{1}(t_{1})T_{2}(t_{2}+\tau e^{-t_{1}})f|_{\tau=0}=
$$
$$
e^{-t_{1}}\partial_{t_{2}}T_{1}(t_{1})T_{2}(t_{2})f.
$$
However,
$$
\partial_{t_{2}}T_{1}(t_{1})T_{2}(t_{2})f=\lim_{\tau\rightarrow 0}T_{1}(t_{1})
\frac{T_{2}(t_{2}+\tau)-I}{\tau}f=
$$
$$
T_{1}(t_{1})\lim_{\tau\rightarrow 0}\frac{T_{2}(t_{2}+\tau)-I}{\tau}f=T_{1}(t_{1})T_{2}(t_{2})\mathbb{A}_{2}f,
$$
and thus 
$$
\partial_{t_{2}}T_{1}(t_{1})T_{2}(t_{2})f=T_{1}(t_{1})T_{2}(t_{2})\mathbb{A}_{2}f,
$$
and then
$$
\mathbb{A}_{2}^{m}T_{1}(t_{1})T_{2}(t_{2})f=e^{-mt_{1}}T_{1}(t_{1})T_{2}(t_{2})\mathbb{A}_{2}^{m}f,\>\>\>f\in \mathcal{G}.
$$
Lemma is proven. 
\end{proof}
Since
$$
\mathbb{A}_{1}^{n}T_{1}(t_{1})T_{2}(t_{2})f=\partial_{t_{1}}^{n}T_{1}(t_{1})T_{2}(t_{2})f,
$$
we obtain the following statement. 
\begin{col}
The following formula holds for $\> n, m\in \mathbb{N}\cup \{0\}$:
\begin{equation}\label{A1A2}
\mathbb{A}_{1}^{n}\mathbb{A}_{2}^{m}T_{1}(t_{1})T_{2}(t_{2})f=e^{-mt_{1}}T_{1}(t_{1})\mathbb{A}_{1}^{n}T_{2}(t_{2})\mathbb{A}_{2}^{m}f,\>\>\>f\in \mathcal{G}.
\end{equation}
\end{col}

\subsection{The Hardy-Steklov-type operators and the Interpolation spaces $\left(\mathbf{ E}, \mathbf{E}^{r}\right)^{K}_{\alpha/r, q}$}\label{Hardy}

For $j=1, 2$ we introduce the Hardy-Steklov-type operators
$$
\mathcal{P}_{j, r}(s)f=(s/r)^{-r}\underbrace{\int_{0}^{s/r}...\int_{0}^{s/r}}_{r}T_{j}( t_{j,1}+...+t_{j,r})fdt_{j,1}...dt_{j,r}, \>\>\>f\in {\bf E},
$$
and the operator $\mathcal{P}_{r}(s)$  which is defined on ${\bf E}$ by the formula 
$$
\mathcal{P}_{r}(s)f=\mathcal{P}_{1,r}(s)\mathcal{P}_{2,r}(s)f=
(s/r)^{-2r}\underbrace{\int_{0}^{s/r}...\int_{0}^{s/r}}_{2r}T_{1}( t_{1,1}+...+t_{1,r})T_{2}( t_{2,1}+...+t_{2,r})f,
$$
where $f\in {\bf E}$ and we dropped the differentials $dt_{1,1}... dt_{2,r}$.
 We are going to prove the next   Lemma.
\begin{lem}\label{smoothing}
 The operator $\mathcal{P}_{r}(s),\>\>r\in \mathbb{N}, \>s\in \mathbb{R},$ is mapping ${\bf E}$  into $ \mathbf{E}^{r}$
 Moreover,  every vector
 $
 \>\mathbb{A}_{j_{1}}...\mathbb{A}_{j_{r}}\mathcal{P}_{r}(s)f,\>\>\>1\leq j_{1}, ... ,j_{r}\leq 2, \>f\in \mathbf{E}, 
 $
 is a linear combination of the vectors
 $$
\mathbb{A}_{1}^{n}\mathbb{A}_{2}^{m}\mathcal{P}_{r}(s) f,\>f\in \mathbf{E}, \>\>\>1\leq n+m\leq r,
$$ 
which are  liner combinations of a terms of the following form
\begin{equation}\label{H}
(s/r)^{-2r}\underbrace{
\int_{0}^{s/r}...\int_{0}^{s/r}}_{2r-n-m}      \xi(s)    T_{1}( t_{1,1}+...+t_{1, r-k})       \left(T_{1}(\tau)-I\right)^{k}F,
\end{equation}
where 
\begin{equation}\label{F}
F=T_{2}( t_{2,1}+...+t_{2,r-m})\left(T_{2}(s/r)-I\right)^{m}f,
\end{equation}
and $\xi(s)$ is a scalar-valued function such that $\>\xi(s)=O(s^{n-k}),\>0\leq k\leq n$.

\end{lem}
\begin{proof}

We note that due to the formula $A_{2}\mathbb{A}_{1}=\mathbb{A}_{1}\mathbb{A}_{2}-\mathbb{A}_{2},$ any product $\mathbb{A}_{j_{1}}...\mathbb{A}_{j_{r}},\>\>1\leq j_{1}, ..., j_{r}\leq 2, m\geq 1, $  can be written as a linear combinations of a products of the form $\mathbb{A}_{1}^{n}\mathbb{A}_{2}^{m},\>\>1\leq n+m\leq m$. According to (\ref{A1A2})
$$
\mathbb{A}_{1}^{n}\mathbb{A}_{2}^{m}T_{1}( t_{1,1}+...+t_{1,r})T_{2}( t_{2,1}+...+t_{2,r})f=
$$
$$
e^{-m(t_{1,1}+...+t_{1,r}   )}T_{1}( t_{1,1}+...+t_{1,r})\mathbb{A}_{1}^{n}T_{2}( t_{2,1}+...+t_{2,r})\mathbb{A}_{2}^{m}f.
$$
By using the formula 
\begin{equation}\label{T2}
\int_{0}^{s/r}T_{2}(\tau)\mathbb{A}_{2}f d\tau=\int_{0}^{s/r}\frac{d}{d\tau}T_{2}(\tau)f d\tau=\left(T_{2}(s/r)-I\right)f,
\end{equation}
$q$ times we obtain
\begin{equation}
(s/r)^{-2r}\underbrace{
\int_{0}^{s/r}...\int_{0}^{s/r}}_{2r}\mathbb{A}_{1}^{n}\mathbb{A}_{2}^{m}T_{1}( t_{1,1}+...+t_{1,r})T_{2}( t_{2,1}+...+t_{2,r})f=
$$
$$
(s/r)^{-2r}\underbrace{
\int_{0}^{s/r}...\int_{0}^{s/r}}_{2r-m}e^{-m(t_{1,1}+...+t_{1,r}   )}
\mathbb{A}_{1}^{n}T_{1}( t_{1,1}+...+t_{1,r})F,
\end{equation}
where $F$ is given in (\ref{F}).
The integration by parts formula gives
\begin{equation}
\int_{0}^{s/r}e^{-m\tau}T_{1}(\tau)\mathbb{A}_{1}Fd\tau=\int_{0}^{s/r}e^{-m\tau}\frac{d}{d\tau}T_{1}(\tau)Fd\tau=
$$
$$
\left(e^{-m s/r}T_{1}(s/r)-I\right)F+m\int_{0}^{s/r}e^{-m\tau}T_{1}(\tau)Fd\tau=
$$
$$
\left(e^{-m s/r}-1\right)F+e^{-m s/r}\left(T_{1}(\tau)-I\right)F+m\int_{0}^{s/r}e^{-m\tau}T_{1}(\tau)Fd\tau.
\end{equation}
By using this formula $p$ times along with the obvious observation that $   \left(e^{-m s/r}-1\right)$  is of order $s$,          we conclude that 
\begin{equation}
(s/r)^{-2r}\underbrace{
\int_{0}^{s/r}...\int_{0}^{s/r}}_{2r}\mathbb{A}_{1}^{n}\mathbb{A}_{2}^{m}T_{1}( t_{1,1}+...+t_{1, r})T_{2}( t_{2,1}+...+t_{2, r})f
\end{equation}
is a liner combination of a terms of the following form
\begin{equation}\label{last-term}
(s/r)^{-2r}\underbrace{
\int_{0}^{s/r}...\int_{0}^{s/r}}_{2r-n-m}      \xi(s)    T_{1}( t_{1,1}+...+t_{1, r-k})       \left(T_{1}(\tau)-I\right)^{k}F,
\end{equation}
where $\xi(s)=O(s^{n-k}),\>1\leq n+m\leq r,\>0\leq k\leq n$.
Since the operators $\mathbb{A}_{1}$ and $\mathbb{A}_{2}$ are closed, our  Lemma is proven.

\end{proof}

For the pair of Banach spaces $\left({\bf E}, \mathbf{E}^{r}\right)$  the $K$-functional  is defined by the formula 
$$
K(s^{r}, f,  {\bf E}, \mathbf{E}^{r})=
$$
$$
\inf_{f=f_{0}+f_{1},\>f_{0}\in {\bf E},
f_{1}\in \mathbf{E}^{r}}
\left(\|f_{0}\|_{{\bf E}}+s^{r}\|f_{1}\|_{\mathbf{E}^{r}}\right).\label{K}
$$
We set
$$
\mathcal{M}_{j,r}(t_{j,1},...,t_{j,r})f=\sum
_{k=1}^{r}(-1)^{k}
C^{k}_{r}T_{j}( k(t_{j,1}+...+t_{j,r})f,
$$
where $C^{k}_{r}$ are the binomial coefficients and introduce
$$ 
\mathcal{H}_{j.r}(s)f=(s/r)^{-r}\int_{0}^{s/r}...\int_{0}^{s/r}\mathcal{M}_{j,r}(t_{j,1},...,t_{j,r})fdt_{j,1}...dt_{j,r}.
$$
An analog of the Hardy-Steklov operator is  defined as follows 
$$
\mathcal{H}_{r}(s)f=
\mathcal{H}_{1,r}(s)\mathcal{H}_{2,r}(s)f,\>\>\>\>f\in \mathbf{E}.
 $$

{\bf Proof of Theorem \ref{Main-Ineq-000}} 
\begin{proof}
We have to show that for every $r\in \mathbb{N}$ there exist a constant $C_{r}$   such that the following inequality holds for all $f\in {\bf E}$
\begin{equation}\label{K-Omega}
c(r)\Omega^{r}(s,f)\leq K(s^{r}, f,  {\bf E}, \mathbf{E}^{r})\leq C(r)\left\{\Omega_{r}(s,f)+\min(s^{r},1)\|f\|_{{\bf E}}\right\}.
\end{equation}
According to Lemma \ref{smoothing} for  $0<s<1$ the following inequality holds
$$
 K\left(s^{r}, f,  {\bf E}, \mathbf{E}^{r}\right) \leq \|f-\mathcal{H}_{r}(s)f\|_{{\bf E}}+s^{r}\|\mathcal{H}_{r}(s)f\|_{\mathbf{E}^{r}}.
$$
We obtain
$$
\|f-\mathcal{H}_{r}(s)f\|_{{\bf E}}\leq 
$$
$$
 (s/r)^{-2r} \left\| \underbrace{    \int_{0}^{s/r} ... \int_{0}^{s/r}   }_{2r}\left[       I-\left(\mathcal{M}_{1,r}(t_{1,1},...,t_{1,r}) \right)\left(\mathcal{M}_{2,r}(t_{2,1},...,t_{2,r})   \right)\right]f
   \right\|_{{\bf E}}.
   $$
   An application of the identity
  \begin{equation}
   1-a_{1}a_{2}=(1-a_{1})+a_{1}(1-a_{2}),
  \end{equation}
   gives
   $$
   \left[       I-\left(\mathcal{M}_{1,r}(t_{1,1},...,t_{1,r}) \right)\left(\mathcal{M}_{2,r}(t_{2,1},...,t_{2,r})   \right)\right]f=
   $$
 $$
      \left[       I-\left(\mathcal{M}_{1,r}(t_{1,1},...,t_{1,r}) \right) \right]f+   \mathcal{M}_{1,r}(t_{1,1},...,t_{1,r}) \left[       I-\left(\mathcal{M}_{1,r}(t_{1,1},...,t_{1,r}) \right)   \right]f,    
   $$
   and then we obtain
\begin{equation}\label{part one}
  \|f-\mathcal{H}_{r}(s)f\|_{{\bf E}}\leq 
  $$
  $$
  c_{0}(r) \left\{ \sup_{0\leq \tau\leq s}\left\|(T_{1}(\tau)-I)^{r}f\right\|_{\bf E}+\sup_{0\leq \tau\leq s}\left\|(T_{2}(\tau)-I)^{r}f\right\|_{\bf E}\right\}\leq 
  c_{0}(r)\Omega^{r}(s, f).
\end{equation}
Next, by Lemma \ref{smoothing} every  term 
$
\>\>s^{r}\mathbb{A}_{j_{1}} ... \mathbb{A}_{j_{r}}\mathcal{H}_{r}(s)f
$
is a linear combination of some terms of the following form
\begin{equation}\label{gen-term}
s^{r}
(s/r)^{-2r}\underbrace{
\int_{0}^{s/r}...\int_{0}^{s/r}}_{2r-n-m}      \xi(s)    T_{1}( t_{1,1}+...+t_{1, r-k})       \left(T_{1}(\tau)-I\right)^{k}F,
\end{equation}
where $\xi(s)=O(s^{n-k}),\>0\leq s\leq 1,$ and $F$ is given by (\ref{F}).
 The norm of a such kind  term  is not greater than
\begin{equation}\label{term-1}
c_{1}(r)s^{r-(m+k)}\sup_{0\leq \tau\leq s}\|(T_{1}(s/r)-I)^{k}T_{2}( \tau)\left(T_{2}(s/r)-I\right)^{m}f\|_{\bf E}.
\end{equation}
Multiple applications of the identity (\ref{first identity-1}) 
allow to estimate the term (\ref{term-1}) by 
$$
c_{2}(r)s^{r-(k_{1}+k_{2})}\sup_{0\leq \tau_{1}, \tau_{2}\leq s}\|(T_{1}(\tau_{1})-I)^{k_{1}}\left(T_{2}(\tau_{2})-I\right)^{k_{2}}f\|_{\bf E}\leq 
$$
$$
c_{2}(r)s^{r-(k_{1}+k_{2})}\Omega^{k_{1}+k_{2}}(s, f).
$$
However, by the inequality (\ref{ineq-3}) the last expression is controlled by 
$
\>\>c_{3}(r)\left\{ s^{r}\|f\|_{\bf E}+\Omega^{r}(s,f)\right\}.
$
Taking in account the inequality (\ref{part one}) we conclude that
$$
 K\left(s^{r}, f,  {\bf E}, \mathbf{E}^{r}\right) \leq \|f-\mathcal{H}_{r}(s)f\|_{{\bf E}}+s^{r}\|\mathcal{H}_{r}(s)f\|_{\mathbf{E}^{r}(\mathbb{A}_{1}, \mathbb{A}_{2})}\leq 
 C(r)\left\{s^{r}\|f\|_{\bf E}+\Omega^{r}(s, f)\right\}.
$$
Thus the right-hand side of (\ref{K-Omega}) is proven.
 According to the inequality (\ref{ineq-1}) one has  for any $f\in \mathbf{E}, \>\>g\in \mathbf{E}^{r}$  the following estimate 
$$
\Omega^{r}(s, f)\leq\Omega^{r}(s, f-g)+\Omega^{r}(s, g)\leq C(r)\left(  \|f-g\|_{\mathbf{E}}+s^{r}\|g\|_{\mathbf{E}^{r}}\right),
$$
which implies the left-hand side of the inequality (\ref{K-Omega}).  Theorem \ref{Main-Ineq-000} is proven.
\end{proof}

{\bf Proof  of Theorem  \ref{Main-000}}

\begin{proof} We will need the following lemma.
\begin{lemma}\label{ineq-00}
The following inequalities hold
\begin{equation}\label{I}
\|f\|_{\mathbf{E}^{k}}\leq C\|f\|_{\mathbf{E}}^{1-k/r}\|f\|_{\mathbf{E}^{r}}^{k/r},\>\>\>f\in \mathbf{E}^{r}, \>\>\>C=C(k,r),
\end{equation}
\begin{equation}\label{II}
K(s^{r},f, \mathbf{E}, \mathbf{E}^{r})\leq      Cs^{r}\|f\|_{\mathbf{E}^{r}},\>\>\>f\in \mathbf{E}^{r}, \>\>\>C=C(r).
\end{equation}

\end{lemma}
\begin{proof}
The first inequality follows from  its well-known version for a single generator of a bounded $C_{0}$-semigroup (see also \cite{Pes22}). The second one follows from the right-hand estimate of (\ref{main-ineq-000}) and (\ref{ineq-1}).
\end{proof}
This lemma shows that one can use the Reiteration Theorem (see \cite{BB}, \cite{KPS}), which immediately implies item (1) of Theorem \ref{Main-000}. Next, let $\alpha>0,\>$ is a non-integer and $[\alpha]$ be its integer part.  According to the equality (\ref{reit}) of Theorem \ref{Main-000} we have 

$$
\left( \mathbf{E}, \mathbf{E}^{r}\right)^{K}_{\alpha/r, q}=\left( \mathbf{E}^{[\alpha]}, \mathbf{E}^{r}\right)^{K}_{(\alpha-[\alpha])/(r-[\alpha]), q}
$$ 
and 
$$
\left( \mathbf{E}, \mathbf{E}^{1}\right)^{K}_{\alpha-[\alpha], q}=\left( \mathbf{E}, \mathbf{E}^{r-[\alpha]}\right)^{K}_{(\alpha-[\alpha])/(r-[\alpha]), q}.
$$
 Note, that $\mathbb{A}_{j_{1}}\mathbb{A}_{j_{2}}...\mathbb{A}_{j_{[\alpha]}},\>\>1\leq j_{1},...,j_{[\alpha]}\leq 2, $ is a continuous map 
$$
\mathbb{A}_{j_{1}}\mathbb{A}_{j_{2}}...\mathbb{A}_{j_{[\alpha]}}: \left( \mathbf{E}^{[\alpha]}, \mathbf{E}^{r}\right)^{K}_{(\alpha-[\alpha])/(r-[\alpha]), q}\mapsto \left( \mathbf{E}, \mathbf{E}^{r-[\alpha]}\right)^{K}_{(\alpha-[\alpha])/(r-[\alpha]), q}.
$$
 All together it shows that if $f\in \left( \mathbf{E}, \mathbf{E}^{r}\right)^{K}_{\alpha/r, q}$ then $\mathbb{A}_{j_{1}}\mathbb{A}_{j_{2}}...\mathbb{A}_{j_{[\alpha]}}f\in \left( \mathbf{E}, \mathbf{E}^{1}\right)^{K}_{\alpha-[\alpha], q}$ and 
\begin{equation}\label{(4)}
\left\|\mathbb{A}_{j_{1}}\mathbb{A}_{j_{2}}...\mathbb{A}_{j_{[\alpha]}}f\right\|_{\left( \mathbf{E}, \mathbf{E}^{1}\right)^{K}_{\alpha-[\alpha], q}}\leq C\|f\|_{ \left( \mathbf{E}, \mathbf{E}^{r}\right)^{K}_{\alpha/r, q}}.
\end{equation}
Conversely,   let $\mathbb{A}_{j_{1}}\mathbb{A}_{j_{2}}...\mathbb{A}_{j_{[\alpha]}}f\in \left( \mathbf{E}, \mathbf{E}^{1}\right)^{K}_{\alpha-[\alpha], q}=\left( \mathbf{E}^{[\alpha]}, \mathbf{E}^{r}\right)^{K}_{(\alpha-[\alpha])/(r-[\alpha]), q}.$ Then the right-hand estimate of (\ref{main-ineq-000}) and (\ref{ineq-1}) imply
\begin{equation}\label{(4')}
\|f\|_{ \left( \mathbf{E}, \mathbf{E}^{r}\right)^{K}_{\alpha/r, q}}\leq C\sum_{1\leq j_{1},...,j_{[\alpha]}\leq 2}\left\|\mathbb{A}_{j_{1}}\mathbb{A}_{j_{2}}...\mathbb{A}_{j_{[\alpha]}}f\right\|_{\left( \mathbf{E}, \mathbf{E}^{1}\right)^{K}_{\alpha-[\alpha], q}}.
\end{equation}
Inequalities (\ref{(4)}) and (\ref{(4')}) imply   item (2) of Theorem \ref{Main-000}. Proof of item (3) is similar. Theorem \ref{Main-000} is completely proved. 

\end{proof}

\section{Paley-Wiener frames in Hilbert space $\bf H$.}\label{Frames}

\subsection{Partitions of unity on the frequency side}\label{subsect:part_unity_freq}

 We keep the notations from the previous sections and assume that $T$ is a unitary representation of the $ "ax+b" $ group $G$ in a Hilbert space ${\bf H}$. As it was mentioned above, the corresponding Laplace operator 
 $$
 \Delta=-\mathbb{A}_{1}^{2}-\mathbb{A}_{2}^{2}
 $$
is self-adjoint and non-negative.

The construction of frequency-localized frames is  achieved via spectral calculus. The idea is to start from a partition of unity on the positive real axis. In the following, we will be considering two different types of such partitions, whose construction we now describe in some detail. The construction below was described in \cite{FFP}.

 Let $g\in C^{\infty}(\mathbb{R}_{+})$ be a non-increasing
 function such that $supp(g)\subset [0,\>  2], $ and $g(\lambda)=1$ for $\lambda\in [0,\>1], \>0\leq g(\lambda)\leq 1, \>\lambda>0.$
 We now let
$
 h(\lambda) = g(\lambda) - g(2 \lambda)~,
$ which entails $supp(h) \subset [2^{-1},2]$,
and use this to define
$$
 F_0(\lambda) = \sqrt{g(\lambda)}~, F_j(\lambda) = \sqrt{h(2^{-j} \lambda)}~, j \ge 1~,
$$
as well as $
 Q_j(\lambda) = \left[F_j(\lambda)\right]^2=F_j^2(\lambda)~, j \ge 0~.$
As a result of the definitions, we get for all $\lambda \ge 0$ the equations
$$
\sum_{j \in \mathbb{N}} Q_j(\lambda) = \sum_{j \in \mathbb{N}} F_j^2(\lambda)
=  g(2^{-n} \lambda),
$$
and as a consequence
$$
\sum_{j \in \mathbb{N}} Q_j(\lambda) = \sum_{j \in \mathbb{N}} F_j^2(\lambda) =  1~,\>\>\>\lambda\geq 0,
$$
 with finitely many nonzero terms occurring in the sums for each
 fixed $\lambda$. 
We call the sequence $(Q_j)_{j \ge 0}$ a {\bf (dyadic) partition of unity}, and $(F_j)_{j \ge 0}$ a {\bf quadratic (dyadic) partition of unity}.  As will become soon apparent, quadratic partitions are useful for the construction of frames.
 Using the spectral theorem one has
$$
F_{j}^{2}(\Delta)  f=\mathcal{F}^{-1}\left(F_{j}^{2}(\lambda)\mathcal{F}f(\lambda)\right),\>\>\>j\geq 1,
$$
and thus
\begin{equation} \label{eqn:quad_part_identity}
 f = \mathcal{F}^{-1}\mathcal{F}f(\lambda) =\mathcal{F}^{-1}\left(\sum_{j \in \mathbb{N}}F_{j}^{2}(\lambda)\mathcal{F}f(\lambda)\right) = \sum_{j \in \mathbb{N}} F_{j}^2(\Delta) f
\end{equation}  
Taking inner product with $f$ gives
$$
\|F_{j}(\Delta) f\|^{2}_{{\bf H}}=\langle F_{j}^{2}(\Delta) f, f \rangle,
$$  
and
\begin{equation}\label{Decomp}
\|f\|_{{\bf H}}^2=\sum_{j \in \mathbb{N}}\langle F_{j}^2 (\Delta) f,f\rangle=\sum_{j \in \mathbb{N}}\|F_{j}(\Delta)f\|_{{\bf H}}^2 .
\end{equation}
Similarly, we get the identity  $
 \sum_{j \in \mathbb{N}} Q_j (\Delta) f = f~.$
Moreover, since the functions $Q_j,  F_{j}$, have their supports in  $
[2^{j-1},\>\>2^{j+1}]$, the elements $ F_{j}(\Delta) f $ and $Q_j (\Delta) f$
 are bandlimited to  $[2^{j-1},\>\>2^{j+1}]$, whenever $j \ge 1$, and to $[0,2]$ for $j=0$.

 \subsection{Paley-Wiener frames in  ${\bf  H}$}\label{Hilb}
We consider the Laplace operator $\Delta$ defined in (\ref{L}) in the Hilbert spaces ${\bf H}$.

\begin{defn}
For every $j\in \mathbb{N}$ let 
$$
\{\Phi_{k}^{j}\}_{k=1}^{K_{j}},\>\>\>\>\>\>\Phi_{k}^{j}\in {\bf PW}_{[2^{j-1}, \>2^{j+1})}\left(\Delta^{1/2}\right),
$$
$$
K_{j}\in \mathbb{N}\cup \{\infty\},
$$ 
 be a frame in ${\bf PW}_{[2^{j-1}, \>2^{j+1})}\left(\Delta^{1/2}\right)$ with the fixed constants $a,\>b$, i.e.
\begin{equation}
a\|f\|_{{\bf H}}^2\leq \sum_{k=1}^{K_{j}}\left|\left< f, \Phi^{j}_{k}\right>\right|^{2}\leq b\|f\|_{{\bf H}}^{2},\>\>\>f \in {\bf PW}_{[2^{j-1}, \>2^{j+1})}\left(\Delta^{1/2}\right).
\end{equation}

\end{defn}
The formula (\ref{Decomp}) and the general theory of frames imply the following statement.

\begin{thm}

\begin{enumerate}
\item The set of functions $\{\Phi_{k}^{j}\}$ will be a frame in the entire space ${\bf H}$ with the same frame constants $a$ and $b$, i.e.
\begin{equation}
a\|f\|_{\bf H}^{2}\leq \sum_{j}\sum_{k}\left|\left< f, \Phi^{j}_{k}\right>\right|^{2}\leq b\|f\|_{\bf H}^{2},\>\>\>f \in {\bf H}.
\end{equation}

\item  The canonical dual frame $\{\Psi^{j}_{k}\}$
also consists of bandlimited  vectors $\Psi^{j}_{k}\in  {\bf PW}_{[2^{j-1},\>2^{j+1}]}(\Delta^{1/2}) ,\>\>j\in  \mathbb{N}, \>k=1,..., K_{j}$, and has the frame bounds $b^{-1},\>\>a^{-1}.$

\item The reconstruction formulas hold for every $f\in {\bf H}$
$$
f=\sum_{j}\sum_{k}\left<f,\Phi^{j}_{k}\right>\Psi^{j}_{k}=\sum_{j}\sum_{k}\left<f,\Psi^{j}_{k}\right>\Phi^{j}_{k}.
$$

\end{enumerate}
\end{thm}

The formula (\ref{Decomp}) implies that in this case the set of functions $\{\Phi_{k}^{j}\}$ will be a frame in the entire ${\bf H}$ with the same frame constants $a$ and $b$, i.e.
\begin{equation}
a\|f\|_{{\bf H}}^2\leq \sum_{j}\sum_{k}\left|\left< f, \Phi^{j}_{k}\right>\right|^{2}\leq b\|f\|_{{\bf H}}^2,\>\>\>f \in {\bf H}.
\end{equation}

\section{More about Besov spaces $\mathbf{ B}^{\sigma}_{ q}$}\label{MoreBes}

In this section we applying Theorems \ref{Direct} and \ref{Inverse} to a situation where $$
{\bf A}={\bf H},\>{\bf B}=\mathbf{H}^{r}, \>\mathbf{B}_{q}^{\alpha}=({\bf H}, \mathbf{H}^{r})^{K}_{\alpha/r, q},\>\>\>
$$
and $\mathcal{T}=\cup_{\omega>0}{\bf PW}_{\omega}\left(\Delta^{1/2}\right)$  
is the abelian additive group  with the quasi-norm
$$
 \| f \|_{\mathcal{T}} = \inf \left \{ \omega'>0~: f \in {\bf PW}_{\mathbf{\omega}'}\left(\Delta^{1/2}\right) \right\}~.
$$
In Lemma \ref{ineq-00}  we proved that the assumptions of Theorems \ref{Direct} and \ref{Inverse} are satisfied. It allows us to formulate the following result.

\begin{thm} \label{approx}
For $\alpha>0, \>\>\>1\leq q\leq\infty,$ the norm of
 $\mathbf{B}_{q}^{\alpha}$,
  is equivalent to
\begin{equation}
\|f\|_{{\bf H}}+\left(\sum_{j=0}^{\infty}\left(2^{j\alpha }\mathcal{E}_{2}(f,
2^{j}; \Delta\right)^{q}\right)^{1/q}.
\end{equation}
\label{maintheorem1}
\end{thm}

Let the functions $F_{j}$ be as in Subsection \ref{subsect:part_unity_freq}.

\begin{thm}\label{projections}
For $\alpha>0, \>\>\>1\leq q\leq\infty,$ the norm of
 $\mathbf{B}_{q}^{\alpha}$,
  is equivalent to

\begin{equation}
f \mapsto \left(\sum_{j=0}^{\infty}\left(2^{j\alpha
}\left \|F_j(\Delta) f\right \|_{{\bf H}}\right)^{q}\right)^{1/q},
\label{normequiv-1}
\end{equation}
  with the standard modifications for $q=\infty$.
\end{thm}

\begin{proof}

We obviously have
$$
\mathcal{E}_{2}(f, 2^{l}; \Delta)\leq \sum_{j> l} \left \|F_j (\Delta) f\right \|_{{\bf H}}.
$$
By using a discrete version of Hardy's inequality \cite{BB} we obtain the estimate
\begin{equation} \label{direct}
\|f\|+\left(\sum_{l=0}^{\infty}\left(2^{l\alpha }\mathcal{E}_{2}(f,
2^{l}; \Delta)\right)^{q}\right)^{1/q}\leq C \left(\sum_{j=0}^{\infty}\left(2^{j\alpha
}\left \|F_j(\Delta) f\right \|_{{\bf H}}\right)^{q}\right)^{1/q}.
\end{equation}
Conversely,
 for any $g\in {\bf PW}_{2^{j-1}}\left(\Delta^{1/2}\right)$ we have
$$
\left\|F_j(\Delta) f\right\|_{{\bf H}}=\left\|F_{j}(\Delta) (f-g)\right\|_{{\bf H}}\leq \|f-g\|_{{\bf H}}.
$$
This implies the estimate
$$
\left\|F_j(\Delta) f\right\|_{{\bf H}}\leq \mathcal{E}_{2}(f,\>2^{j-1}; \Delta),
$$
which shows that the inequality opposite to (\ref{direct}) holds.
 The proof is complete.
\end{proof}

 \begin{thm}\label{framecoef}
For $\alpha>0,\>\>\> 1\leq q\leq\infty,$ the norm of
 $\mathbf{B}_{q}^{\alpha}$
  is equivalent to
\begin{equation}
 \left(\sum_{j=0}^{\infty}2^{j\alpha q }
\left(\sum_{k}\left|\left<f,\Phi^{j}_{k}\right>\right|^{2}\right)^{q/2}\right)^{1/q}\asymp \|f\|_{\mathbf{B}_{q}^{\alpha}},
\label{normequiv}
\end{equation}
  with the standard modifications for $q=\infty$.
\end{thm}

\begin{proof}
For   $f\in {\bf  H}$  and the operator $F_{j}(\Delta)$
we have
\begin{equation}
a\left \|F_j (\Delta) f\right \|_{{\bf H}}^{2}\leq
\sum_k\left|\left< F_{j} (\Delta) f, \phi^{j}_{k}\right>\right|^{2}\leq b
\left\|F_j (\Delta) f\right \|_{\bf H}^{2},
\end{equation}
 and then 
 we obtain the following inequality
$$
\sum_{k}\left|\left<f,\Phi^{j}_{k}\right>\right|^{2}\leq \left \|F_j (\Delta) f\right \|_{{\bf H}}^{2}\leq a^{-1}
\sum_{k}\left|\left<f,\Phi^{j}_{k}\right>\right|^{2}
,\>\>f\in {\bf H}.
$$
Theorem is proven.
\end{proof}

\section{Appendix. $K$-functional, Interpolation and Approximation spaces}\label{Appendix}

The goal of the section is to introduce basic notions of the theory of interpolation spaces  \cite{BB},  \cite{BL}, \cite{KPS}, \cite{T}, and  approximation spaces  \cite{BL},  \cite{KP}, \cite{PS}.  
It is important to realize that the relations between
interpolation and approximation spaces cannot be described
in  the language of normed spaces. We have to make use of
quasi-normed linear spaces in order to treat them
simultaneously.

A quasi-norm $\|\cdot\|_{\bf A}$ on a linear space $\bf A$ is
a real-valued function on $\bf E$ such that for any
$f,f_{1}, f_{2} \in {\bf A}$ the following holds true:  
\begin{enumerate}

\item $\|f\|_{\bf A}\geq 0;\>\>\>$

\item $\|f\|_{\bf A}=0  \Longleftrightarrow   f=0;\>\>\>$

\item $\|-f\|_{\bf A}=\|f\|_{\bf A};\>\>\>$

\item there exists some $C_{\EB} \geq 1$ such that
$\|f_{1}+f_{2}\|_{\EB}\leq C_{\EB}(\|f_{1}\|_{\EB}+\|f_{2}\|_{\EB}).\>\>$

\end{enumerate}
Two quasi-normed linear spaces $\EB$ and $\FB$ form a
pair if they are linear subspaces of a common linear space
$\mathcal{A}$ and the conditions
$\|f_{k}-g\|_{\EB}\rightarrow 0,$ and
$\|f_{k}-h\|_{\FB}\rightarrow 0$
imply equality $g=h$ (in $\mathcal{A}$).
For any such pair $\EB,\FB$ one can construct the
space $\EB \cap \FB$ with quasi-norm
$$
\|f\|_{\EB \cap \FB}=\max\left(\|f\|_{\EB},\|f\|_{\FB}\right)
$$
and the sum of the spaces,  $\EB + \FB$ consisting of all sums $f_0+f_1$ with $f_0 \in \EB, f_1 \in \FB$, and endowed with the quasi-norm
$$
\|f\|_{\EB + \FB}=\inf_{f=f_{0}+f_{1},f_{0}\in \EB, f_{1}\in
\FB}\left(\|f_{0}\|_{\EB}+\|f_{1}\|_{\FB}\right).
$$

Quasi-normed spaces $\HB$ with
$\EB \cap \FB \subset \HB \subset \EB + \FB$
are called intermediate between $\EB$ and $\FB$.
If both $E$ and $F$ are complete the
inclusion mappings are automatically continuous. 
An additive homomorphism $T: \EB \rightarrow \FB$
is called bounded if
$$
\|T\|=\sup_{f\in \EB,f\neq 0}\|Tf\|_{\FB}/\|f\|_{\EB}<\infty.
$$
An intermediate quasi-normed linear space $\HB$
interpolates between $\EB$ and $\FB$ if every bounded homomorphism $T:
\EB+\FB \rightarrow \EB + \FB$
which is a bounded homomorphism of $\EB$ into
$\EB$ and a bounded homomorphism of $\FB$ into $\FB$
is also a bounded homomorphism of $\HB$ into $\HB$.
On $\EB+\FB$ one considers the so-called Peetre's $K$-functional
$$
K(f, t)=K(f, t,\EB, \FB)=\inf_{f=f_{0}+f_{1},f_{0}\in \EB,
f_{1}\in \FB}\left(\|f_{0}\|_{\EB}+t\|f_{1}\|_{\FB}\right).\label{K}
$$
The quasi-normed linear space $(\EB,\FB)^{K}_{\theta,q}$,
with parameters $0<\theta<1, \,
0<q\leq \infty$,  or $0\leq\theta\leq 1, \, q= \infty$,
is introduced as the set of elements $f$ in $\EB+\FB$ for which
\begin{equation}
\|f\|_{\theta,q}=\left(\int_{0}^{\infty}
\left(t^{-}K(f,t)\right)^{q}\frac{dt}{t}\right)^{1/q} < \infty .\label{Knorm}
\end{equation}

It turns out that $(\EB,\FB)^{K}_{\theta,q}$
with the quasi-norm
(\ref{Knorm})  interpolates between $\EB$ and $\FB$.

Let us introduce another functional on $\EB+\FB$,
where $\EB$ and $\FB$ form a pair of quasi-normed linear spaces
$$
\mathcal{E}(f, t)=\mathcal{E}(f, t, \mathbf{E},  \mathbf{F})=\inf_{g\in \FB,
\|g\|_{\FB}\leq t}\|f-g\|_{\EB}.
$$
\begin{defn}
The approximation space $\mathcal{E}_{\alpha,q}(\EB, \FB),
0<\alpha<\infty, 0<q\leq \infty $ is the quasi-normed linear spaces
of all $f\in \EB+\FB$ for which the quasi-norm
\begin{equation}
\| f \|_{\mathcal{E}_{\alpha,q}(\EB, \FB)} = \left(\int_{0}^{\infty}\left(t^{\alpha}\mathcal{E}(f,
t)\right)^{q}\frac{dt}{t}\right)^{1/q}
\end{equation}
is finite.
\end{defn}

The next two theorems   represent a very abstract version of what is  known as a Direct and an Inverse  Approximation Theorems \cite{PS,BS}. In the form it is stated below they were proved in \cite{KP}.

\begin{thm}\label{Direct}
 Suppose that $\mathcal{T}\subset \FB \subset \EB$ are quasi-normed
linear spaces and $\EB$ and $\FB$ are complete.
If there exist $C>0$ and $\beta >0$ such that
the following Jackson-type inequality is satisfied
$
t^{\beta}\mathcal{E}(t,f,\mathcal{T},\EB)\leq C\|f\|_{\FB},
\>\>t>0, \>\> f \in \FB,$
 then the following embedding holds true
\begin{equation}\label{imbd-1}
(\EB,\FB)^{K}_{\theta,q}\subset
\mathcal{E}_{\theta\beta,q}(\EB, \mathcal{T}), \quad \>0<\theta<1, \>0<q\leq \infty.
\end{equation}
\end{thm}

\begin{thm}\label{Inverse}
If there exist $C>0$ and $\beta>0$
such that
the following Bernstein-type inequality holds
$
\|f\|_{\FB}\leq C\|f\|^{\beta}_{\mathcal{T}}\|f\|_{\EB}
,\>\> f\in \mathcal{T},
$
then the following embedding holds true
\begin{equation}\label{imbd-2}
\mathcal{E}_{\theta\beta, q}(\EB, \mathcal{T})\subset
(\EB, \FB)^{K}_{\theta, q}  , \quad 0<\theta<1, \>0<q\leq \infty.
\end{equation}
\end{thm}

\bibliographystyle{amsalpha}

\end{document}